\documentclass[11pt]{article}

\usepackage{amsmath,amssymb,amsthm}
\usepackage{authblk,hyperref}
\usepackage[utf8]{inputenc}

\usepackage{tikz}
\usetikzlibrary{matrix,arrows}

\usepackage[top=3cm,bottom=3cm,left=3cm,right=3cm]{geometry}
\usepackage{setspace}
\theoremstyle{plain}
\newtheorem{theorem}{Theorem}

\newtheorem{corollary}[theorem]{Corollary}
\newtheorem{lemma}[theorem]{Lemma}

\theoremstyle{definition}
\newtheorem{definition}[theorem]{Definition}

\renewcommand{\theta}{\vartheta}
\renewcommand{\phi}{\varphi}

\newif\ifshowproofs
\showproofsfalse

\newcommand{\CD}{\mathbf{CD}}
\newcommand{\G}{\mathbf{G}}
\renewcommand{\a}{\overline{a}}
\newcommand{\x}{\overline{x}}
\newcommand{\forces}{\Vdash}

\begin{document}
\title{An interpolant in predicate  Gödel logic}
\author{Matthias Baaz$^1$ \ \href{mailto:baaz@logic.at}{baaz@logic.at}, \\
Mai Gehrke$^2$ \ \href{mailto:mgehrke@unice.fr}{mgehrke@unice.fr},\\
 Sam van Gool$^3$ \ \href{mailto:samvangool@me.com}{samvangool@me.com}}
\affil{$^1$ Institute of Discrete Mathematics and Geometry, TU Wien, Austria \\
\hspace*{-2.4cm}$^2$ Laboratoire J. A. Dieudonn\'e, CNRS {\&} UNSA, France\\
\hspace*{-1.8cm}$^3$ Institute for Logic, Language and Computation, UvA, NL}


\maketitle

\begin{abstract}
A logic satisfies the interpolation property provided that whenever a formula $\Delta$ is a consequence of another formula $\Gamma$, then this is witnessed by a formula $\Theta$ which only refers to the language common to $\Gamma$ and $\Delta$. That is, the relational (and functional) symbols occurring in $\Theta$ occur in both $\Gamma$ and $\Delta$, $\Gamma$ has $\Theta$ as a consequence, and $\Theta$ has $\Delta$ as a consequence. Both classical and intuitionistic predicate logic have the interpolation property, but it is a long open problem which intermediate predicate logics enjoy it. In 2013 Mints, Olkhovikov, and Urquhart showed that constant domain intuitionistic logic does not have the interpolation property, while leaving open whether predicate G\"odel logic does. In this short note, we show that their counterexample for constant domain intuitionistic logic does admit an interpolant in predicate G\"odel logic. While this has no impact on settling the question for predicate G\"odel logic, it lends some credence to a common belief that it does satisfy interpolation. Also, our method is based on an analysis of the semantic tools of Olkhovikov and it is our hope that this might eventually be useful in settling this question.
\end{abstract}

\section{Introduction}

The interpolation property has been studied intensively over the years for both predicate and propositional logics. Indeed, interpolation was proved for classical predicate logic by Craig \cite{Cr57} in 1957 and for intuitionistic predicate logic by Schütte \cite{Sc62}. In 1977  Maksimova \cite{Ma77} completely solved the interpolation problem for intermediate propositional logics by showing that exactly 7 of these logics have the interpolation property. Her work uses the algebraic semantics available for these propositional logics. In the setting of predicate logics, algebraic semantics are not as well understood and the question still remains open for many intermediate predicate logics despite the fact that the question has been actively pursued. Recent advances on the subject include the 2013 paper by Mints, Olkhovikov, and Urquhart \cite{MOU13} in which they show that constant domain intuitionistic logic does not have the interpolation property and the very recent contribution \cite{BaLo17} showing, among other, that constant domain intermediate logics based on finite algebras of truth values as well as some fragments of G\"odel logic do have the interpolation property. However, the question remains open for full predicate G\"odel logic, the logic of all linearly ordered constant domain Kripke models.

Consider the following two formulas:
\begin{align*}
\Gamma &: \forall x \exists y (Py \wedge (Qy \to Rx)) \wedge \neg \forall x Rx, \\
\Delta &: \forall x (Px \to (Qx \vee S)) \to S.
\end{align*}
It was proved in \cite{MOU13} that $\Gamma \to \Delta$ is valid in constant domain intuitionistic predicate logic ($\CD$), but that there does not exist an interpolant for $\Gamma \to \Delta$ in the common language containing only the predicate symbols $P$ and $Q$.

We will show that the example from \cite{MOU13} does not provide a counterexample to interpolation for predicate Gödel logic, $\G$. In particular, we show that
\begin{align*}
\Theta &: \forall x (\neg Px \vee \exists y (Py \wedge (Qy \to Px))) \wedge \neg \forall x (\neg Px \vee Qx)
\end{align*}
is an interpolant for $\Gamma \to \Delta$ in $\G$. That is, we will show (Theorem~\ref{thm:main}) that $\Gamma \to \Theta$ and $\Theta \to \Delta$ are both valid in $\G$. (In fact, we will show that $\Theta \to \Delta$ is even valid in $\CD$.)

Our interpolant came about by extracting a formula from our analysis of the proof of the counterexample in \cite{MOU13}, showing that a pair of countermodels such as those exhibited in \cite{MOU13} necessarily involve models which are \emph{not} linearly ordered. One may hope that a generalisation of these ideas, along with a strengthening of the completeness result for the semantic tools of Olkhovikov \cite{Ol14} may provide a route to a resolution of this longstanding question.

\section{Semantics}
We recall the constant domain semantics for $\CD$ and $\G$.
\begin{definition}
Let $\mathcal{L}$ be a finite set of unary predicate symbols.\footnote{We only need to consider unary predicates in this note.} A \emph{model} (over $\mathcal{L}$) is a tuple $M = (W,\leq,w_0,A,(P^W)_{P \in \mathcal{L}})$, where $(W,\leq)$ is a quasi-order, the \emph{base point} $w_0$ is an element of $W$ such that $w_0 \leq w$ for all $w \in W$, $A$ is a set, and for each $P \in \mathcal{L}$, $P^W$ is an order-preserving function from $W$ to $\mathcal{P}(A)$, i.e., if $w \leq w'$ in $W$, then $P^W(w) \subseteq P^W(w')$. We will often suppress the superscript $W$ in $P^W$ when no confusion can arise.

For any model $M = (W,\leq,w_0,A,(P^W))$ and finite set of variables $X$, the \emph{forcing relation}, $\forces$, at a world $w$ and an assignment $\a \colon X \to A$, is defined by induction on the complexity of formulas $\phi$ all of whose free variables lie in $X$, as follows:
\begin{itemize}
\item $w, \a \forces Px$ iff $\a(x) \in P^W(w)$;
\item $w, \a \forces \phi \vee \psi$ iff $w, \a \forces \phi$ or $w, \a \forces \psi$;
\item $w, \a \forces \phi \wedge \psi$ iff $w, \a \forces \phi$ and $w, \a \forces \psi$;
\item $w, \a \forces \phi \to \psi$ iff for all $w' \geq w$, if $w', \a \forces \phi$ then $w', \a \forces \psi$;
\item $w, \a \not\forces \bot$;
\item $w, \a \forces \exists y \phi$, where $y\not\in X$, if and only if there exists $b \in A$ such that $w, \a \cup \{(y,b)\} \forces \phi$;
\item $w, \a \forces \forall y \phi$, where $y\not\in X$, if and only if for all $b \in A$, we have $w, \a \cup \{(y,b)\} \forces \phi$.
\end{itemize}
As usual, $\neg \phi$ is regarded as an abbreviation of $\phi \to \bot$, so that the semantic clause becomes:
\begin{itemize}
\item $w, \a \forces \neg \phi$ iff for all $w' \geq w$, $w', \a \not\forces \phi$.
\end{itemize}

By a slight abuse of notation, if $\phi(\x)$ is a formula, and $\a \colon \x \to A$ is an assignment, then we also write $w \forces \phi(\a)$ instead of $w, \a \forces \phi(\x)$. For example, if $a \in A$ and $w \in W$, then $w \forces Pa$ means $a \in P^W(w)$.

A \emph{linear model} is a model such that, for any two formulas $\phi$, $\psi$, the instance of the scheme $(\phi \to \psi) \vee (\psi \to \phi)$ is forced in the base point $w_0$ under every assignment. Up to logical equivalence of models, this means that we may assume that the quasi-order $\leq$ is total.
\end{definition}
It is straight forward to establish that any model is \emph{persistent}: for any formula $\phi(\x)$ and assignment $\a \colon \x \to A$, if $w \leq w'$ and $w,\a \forces \phi$ then $w',\a \forces \phi$. Therefore, if a formula is forced in the base point, it is forced everywhere in the model. We will use the following completeness theorem for $\G$ \cite{Co92}.
\begin{theorem}\label{thm:completeness}
For any sentence $\phi$ using only predicate symbols from $\mathcal{L}$, the following are equivalent:
\begin{enumerate}
\item the sentence $\phi$ is valid in $\G$;
\item for any linear model $M$ with base point $w_0$, we have $w_0 \forces \phi$.
\end{enumerate}
\end{theorem}
\begin{corollary}\label{cor:completeness}
Suppose that, for any linear model $M$ with base point $w_0$ such that $w_0 \forces \phi$, we have $w_0 \forces \psi$. Then $\phi \to \psi$ is valid in $\G$.
\end{corollary}
\begin{proof}
Let $M$ be a linear model with base point $w_0$. Let $w \geq w_0$ be an arbitrary point in which $\phi$ is forced. Then the restriction of $M$ to a model $M'$ on the set ${\uparrow} w$ is a linear model with base point $w$, and $\phi$ is still forced in $w$ in $M'$. By assumption, $w \forces \psi$. Thus, by definition of the semantic clause for $\to$, $w_0 \forces \phi \to \psi$. Since $M$ was arbitrary, $\phi \to \psi$ is valid in $\G$ by Theorem~\ref{thm:completeness}.
\end{proof}

\section{Interpolant}
Given this completeness theorem, we will now establish that $\Gamma \to \Theta$ and $\Theta \to \Delta$ are valid in $\G$ by checking that these two formulas hold in all linear models. In fact, we will see that $\Theta \to \Delta$ holds in all models, and is hence valid in $\CD$. By the main result of \cite{MOU13}, $\Gamma \to \Theta$ is not valid in $\CD$.

We recall an important lemma from \cite{MOU13}. It characterises the validity of  second-order formulas $\exists R \Gamma$ and $\forall S \Delta$ on a model in the language $\{P,Q\}$ as first-order properties of that model. In what follows, if $M = (W,\leq,w_0,A,(P^W)_{P \in \mathcal{L}})$ is a model over a language $\mathcal{L}$, $R$ is a symbol not in $\mathcal{L}$, and $R^W \colon W \to \mathcal{P}(A)$ is order-preserving, then we will write $(M,R^W)$ for the expanded model $M' = (W,\leq,w_0,A,(P^W)_{P \in \mathcal{L} \cup \{R\}})$.
\begin{lemma}\label{lem:fo-char}
Let $M = (W,\leq,w_0,A,P^W,Q^W)$ be a model over $\mathcal{L} = \{P,Q\}$ with base point $w_0$.
\begin{enumerate}
\item The following are equivalent:
\begin{enumerate}
\item There exists order-preserving $R^W \colon W \to \mathcal{P}(A)$ such that, in the expanded model $(M,R^W)$, we have $w_0 \forces \Gamma$;
\item For every $w \in W$, there exists $a \in A$ such that $w_0 \forces Pa$ and $w \not\forces Qa$.
\end{enumerate}
\item The following are equivalent:
\begin{enumerate}
\item For every order preserving $S^W \colon W \to \mathcal{P}(A)$, in the expanded model $(M,S^W)$, we have $w_0 \forces \Delta$;
\item For every $w \in W$, there exists $a \in A$ such that $w \forces Pa$ and $w\not\forces Qa$.
\end{enumerate}
\end{enumerate}
\end{lemma}
\begin{proof}
See \cite[Lemma 4.2]{MOU13}.
\end{proof}

We are now ready to prove that $\Theta$ is an interpolant.
\begin{theorem}\label{thm:main}
Let $\Gamma, \Theta$ and $\Delta$ be the formulas defined above.

\begin{enumerate}
\item The implication $\Gamma \to \Theta$ is valid in $\G$.
\item The implication $\Theta \to \Delta$ is valid in $\CD$.
\end{enumerate}
\end{theorem}
\begin{proof}
1. We establish the sufficient condition from  Corollary~\ref{cor:completeness}. Let $M$ be a linear model over $\{P,Q,R\}$ with base point $w_0$ such that $w_0 \forces \Gamma$. We need to show that $w_0 \forces \Theta$, that is: (a) $w_0 \forces \forall x (\neg Px \vee \exists y (Py \wedge (Qy \to Px)))$ and (b) $w_0 \forces  \neg \forall x (\neg Px \vee Qx)$.  

\begin{enumerate}
\item[(a)] Let $a \in A$ be arbitrary. We show that $w_0 \forces \neg Pa \vee \exists y (Py \wedge (Qy \to Pa))$. If $w_0 \forces \neg Pa$, we are done immediately. Suppose that $w_0 \not\forces \neg Pa$. We show that $w_0 \forces \exists y(Py \wedge (Qy \to Pa))$. Since $w_0 \not\forces \neg Pa$, pick $w_1 \geq w_0$ such that $w_1 \forces Pa$. By Lemma~\ref{lem:fo-char}.1, pick $b$ such that $w_0 \forces Pb$ and $w_1 \not\forces Qb$. Then $w_0 \not\forces Pa \to Qb$, since $w_1 \geq w_0$ and $w_1 \forces Pa$, but $w_1 \not\forces Qb$. Since $M$ is a linear model $w_0 \forces (Pa \to Qb) \vee (Qb \to Pa)$, so we conclude that $w_0 \forces Qb \to Pa$. Since also $w_0 \forces Pb$, we have proved that $w_0 \forces Pb \wedge (Qb \to Pa)$, so that $w_0 \forces \exists y(Py \wedge (Qy \to Pa))$.

\item[(b)] Let $w \in W$ be arbitrary. By Lemma~\ref{lem:fo-char}.1, pick $a$ such that $w_0 \forces Pa$ and $w \not\forces Qa$. Then, since $w_0 \leq w$, we have $w \forces Pa$. Thus, $w \not\forces \neg Pa \vee Qa$. Hence, $w \not\forces \forall x (\neg Px \vee Qx)$. Since $w$ was arbitrary, we get that $w_0 \forces \neg \forall x (\neg Px \vee Qx)$.
\end{enumerate}

2. We establish the sufficient condition from the analogous version of Corollary~\ref{cor:completeness} for $\CD$. Let $M$ be a model over $\{P,Q,S\}$ with base point $w_0$ such that $w_0 \forces \Theta$. We need to show that $w_0 \forces \Delta$. By Lemma~\ref{lem:fo-char}.2, it suffices to prove that:
\begin{align}\tag{*}
\text{for every $w$, there exists $c \in A$ such that $w \forces Pc$ and $w\not\forces Qc$.}
\end{align}
Let $w$ be arbitrary. Since $w_0 \forces \neg \forall x(\neg Px \vee Qx)$, in particular we have that $w \not\forces \forall x (\neg Px \vee Qx)$. Pick $a \in A$ such that $w \not\forces \neg Pa \vee Qa$, that is, $w \not\forces \neg Pa$ and $w \not\forces Qa$. Since $w_0 \forces \Theta$, in particular $w \forces \Theta$, and instantiating the first conjunct with $x = a$, we see that $w \forces \neg Pa \vee \exists y (Py \wedge (Qy \to Pa))$. Since $w \not\forces \neg Pa$, we conclude that $w \forces \exists y (Py \wedge (Qy \to Pa))$. Pick $b \in A$ such that $w \forces Pb$ and $w \forces Qb \to Pa$. We distinguish two cases: (i) $w \forces Qb$. In this case, since $w \forces Qb \to Pa$, we have $w \forces Pa$. Since also $w \not\forces Qa$, we can take $c := a$ in (*). (ii) $w \not\forces Qb$. In this case, we can take $c := b$ in (*).  This establishes (*), so that $w_0 \forces \Delta$.
\end{proof}

\subsection*{Acknowledgement}
The authors gratefully acknowledge the support of the project DuaLL, funded by the European Research Council under the Horizon 2020 program.



\end{document}